\documentclass[12pt]{article}   	
\usepackage{geometry}                		
\usepackage{graphicx}				
\usepackage{caption}
\usepackage{subcaption}
\usepackage{amssymb}
\usepackage{amsthm}
\usepackage{epsfig}
\usepackage{csquotes}
\usepackage{amsmath}
\usepackage{mathtools}
\usepackage{float}
\usepackage{physics}
\MakeOuterQuote{"}
\newtheorem*{theorem*}{Theorem}

\newtheorem{theorem}{Theorem}[section]

\newtheorem{lemma}[theorem]{Lemma}

\newtheorem{corollary}[theorem]{Corollary}

\newtheorem*{conjecture*}{Conjecture}
\title{The Affine Shape of a Figure-Eight under the Curve Shortening Flow}
\author{Matei P. Coiculescu \thanks{Supported by an N.S.F. Graduate Research Fellowship}
\ and Richard Evan Schwartz \thanks{Supported by an N.S.F. Research Grant (DMS-1807320)}}
\begin{document}
\maketitle

\begin{abstract}
We consider the curve shortening
flow applied to a class of
figure-eight curves:
those with dihedral symmetry,
convex lobes,
and a monotonicity assumption on the
curvature.
We prove that when (non-conformal)
linear transformations are applied
to the solution so as to keep the
bounding box the unit square, the
renormalized limit converges to a quadrilateral
$\Join$ which we call a bowtie.  Along the way
we prove that  suitably chosen arcs
of our evolving curves, when suitably rescaled,
converge to the Grim Reaper Soliton
under the flow. Our Grim Reaper Theorem is
an analogue of a theorem of S. Angenent in \cite{ANG},
which is proven in the locally convex case.
\end{abstract}

\section{Introduction}

We say that a smooth family
$C: S^1\cross [0,T) \to {\bf R\/}^2$ of closed immersed plane curves
is evolving according to {\it curve shortening flow\/} (CSF) if and only 
if for  any point $(u,t) \in S^1\cross [0,T)$ we have
$$\frac{\partial C}{\partial t} = k N$$
where $k$ is the curvature and $N$ is the unit normal vector of the immersed 
curve $u \to C(u,t)$.  We often abbreviate this curve as $C(t)$.
In all cases, there is some time $T>0$, called the
{\it vanishing time\/}, such that
$C(t)$ is defined for all $t \in (0,T)$ but
not at time $T$.

Some powerful results are known about this PDE.
In \cite{GH}, M. Gage and R. Hamilton prove that
when $C(0)$ is convex the curve
$C(t)$ (which remains convex) shrinks to a point
as $t \to T$ and, moreover, there is a similarity
$S_t$ such that
$S_t(C(t))$ converges to the unit circle.
 See also \cite{G1} and \cite{G2}.
In \cite{G}, M. Grayson proves that if
$C(0)$ is embedded then there is some time
$t \in (0,T)$ such that $C(t)$ is convex.
Thus, the combination of these two results says
informally that the curve-shortening flow shrinks embedded
curves to round points.

In \cite{ANG2} and \cite{ANG3}, S. Angenent proves that
if $C(0)$ is immersed and with finitely many
self-intersections, then
the number of self-intersections 
is monotone non-increasing with time.
In the case of a Figure-$8$, a smooth loop
with one self-intersection, M. Grayson
proves two things \cite{G3}:
\begin{itemize}
\item If one of the two lobes of the figure-eight has
smaller area than the other, then this lobe shrinks
to a point before the vanishing time.  Then the flow
can be continued through the singularity and it turns into
the embedded case.
\item If the lobes have equal area, the double point
    does not disappear before the vanishing time $T$, and the
  isoperimetric ratio of $C(t)$ tends to $\infty$ as
  $t \to T$.
\end{itemize}

Grayson conjectures \cite{G3} that in the second case,
the figure-$8$ converges to a point
under the curve-shortening flow, but this
is as yet unresolved. In case
$C(0)$ has $2$-fold rotational symmetry,
it follows from Corollary 2 of  \cite{DHW} that
$C(t)$ does shrink to a point (the double point) as $t \to T$. 
In a related direction, the papers
\cite{AL}, \cite{EW}, and \cite{HH} discuss
self-similar solutions to the CSF.
These shrink to a point and retain their shape.

We work with figure-$8$ curves that have
convex lobes and $4$-fold dihedral symmetry.
We normalize so that the coordinate
axes are the symmetry axes and that the $x$-axis
intersects the curve in $3$ points.  Thus,
our curves look like $\infty$ symbols.
Angenent proves in \cite{ANG2} and \cite{ANG3} that
if $C(0)$ has convex lobes then so does
$C(t)$ for all $t \in (0,T)$.

Let $C_+(t)$ denote the righthand lobe of $C(t)$.
We define $\kappa(\theta,t)>0$ to be the curvature of
$C_+(t)$ at the point where the tangent line makes an angle
$\theta$ with the $x$-axis.  We measure this angle
in such a way that the top half of $C_+(t)$ is
parametrized by $\theta \in (-\alpha(t),\pi/2]$, where $\alpha(t)$ is
the tangent angle at the origin.
Let $\kappa_{\theta}=\partial \kappa/\partial \theta$, etc.
Computing the time evolution of $\kappa$, we have
\begin{equation}
    \label{MAIN}
    \kappa_t=\kappa^2(\kappa+\kappa_{\theta \theta}).
\end{equation}
See \cite{GH} for a proof. We note that the
curve satisfying $\kappa(\theta)=\sin(\theta)$,
for $\theta \in (0,\pi)$ is a
stationary solution for Equation \ref{MAIN}.  Up to isometries of the plane,
this curve is called the {\it Grim Reaper Soliton\/}. It evolves by
translation under the curve shortening flow.
\\\\
\noindent
{\bf Definition:\/}
$C(0)$ is a {\it monotone\/} figure-eight curve if and only if
\begin{itemize}
  \item $C(0)$ is real analytic.
  \item $C(0)$ has $4$-fold dihedral symmetry.
  \item $C(0)$ has convex lobes.
  \item $\kappa_{\theta}(\theta,0)>0$ for $\theta \in
    (-\alpha(0),\pi/2)$
  \item $\kappa_{\theta \theta}(\pi/2,0) \not =0$.
  \item  The signed curvature of $C(0)$, as a function of arc length, does
    not vanish to second order at the double point.
  \end{itemize}

  The Lemniscate of Bernoulli is an example of a monotone figure 8
  curve.  The first condition is not much of a restriction because the
  curve shortening flow instantly turns curves real analytic.
  The last two conditions are nondegeneracy conditions
  included to simplify our arguments.
In \S 2 we prove that the curve-shortening flow preserves monotonicity:
if $C(0)$ is monotone, then $C(t)$ is monotone for all $t<T$.
    The proof is basically an application of the maximum principle
    and the so-called Sturmian principle for various strictly
    parabolic
    equations.
  
Define
\begin{equation}
  F(\theta,t)=\frac{\kappa(\theta,t)}{\kappa(\pi/2,t)}.
  \end{equation}
  Here $F$ is a rescaled version of $\kappa$.
In \S 3 we prove the following result.
  
\begin{theorem}[Grim Reaper]
  Assume that $C(0)$ is monotone.
  Let $J \subset (0,\pi)$ be an arbitrary closed
  interval. Let $\epsilon>0$ be given.
  For $t$ sufficiently close to $T$, we have
  $$\sup_{\theta \in J} |F(\theta,t)-\sin(\theta)|<\epsilon, \hskip 30pt
  \sup_{\theta \in J} |F_{\theta}(\theta,t)-\cos(\theta)|<\epsilon.$$
\end{theorem}
The Grim Reaper Theorem is the analogue of Theorem D in
\cite{ANG}.  
  In \cite{ANG}, S. Angenent also makes the monotonicity
  assumption when applying his Theorem D to specific curves.
The result implies that a suitably rescaled
copy of the arc of $C(t)$ corresponding to
$\theta \in (0,\pi)$ converges to the Grim Reaper curve.
The arc in question is the one between the two
dots in Figure 1.
Our proof departs from that in
\cite{ANG} because we are not working with
locally convex curves as in \cite{ANG}.

The {\it bounding box\/} of a
compact set in the plane is the smallest rectangle,
with sides parallel to the coordinate axes,
that contains the set.
The main goal of the paper is to
understand the limit of the curves
$\{L_t(C(t)\}$ where $L_t$ is the positive
diagonal matrix such that $L_t(C(t))$
has the square $[-1,1]^2$ for a bounding box.
Even though affine transformations do not interact
in a nice way with the curve shortening flow,
nothing stops us from looking at a solution
and applying affine transformations afterwards.

The {\it bowtie\/} is the quadrilateral
whose vertices are
$$(-1,-1), \hskip 15 pt (1,1), \hskip 15 pt (1,-1), \hskip 15 pt (-1,1)$$
in this cyclic order.   It is shaped like this: $\Join$.
The {\it Hausdorff distance\/} between two
compact subsets of the plane is the smallest
$\epsilon$ such that each set is contained
in the $\epsilon$-neighborhood of the other.
This distance makes the set of compact
planar subsets into a metric space.
Here is our main result.

\begin{theorem}[Bowtie]
\label{bowtie2}
Suppose that $C(0)$ is monotone.
  As $t \to T$, the curves
  $L_t(C(t))$ converge in the Hausdorff metric to the
  bowtie.
\end{theorem}

Figure 1 shows a picture of a numerical simulation of the
curve shortening flow.  The curve on the left
is $L_0(C(0))$ where $C(0)$ is the Lemniscate of
Bernoulli.
The black curve on the right is $L_t(C(t))$
for some later time $t$.
The blue curve on the right is $\Gamma(t)=C(t)/X(t)$, the
rescaled version of $C(t)$ whose bounding box has
width $2$.
The black and white dots correspond 
to where $\theta=0$ and $\theta=\pi$
respectively. Figure 1 shows
some hints of the bowtie forming.

\begin{center}
\resizebox{!}{2.5in}{\includegraphics{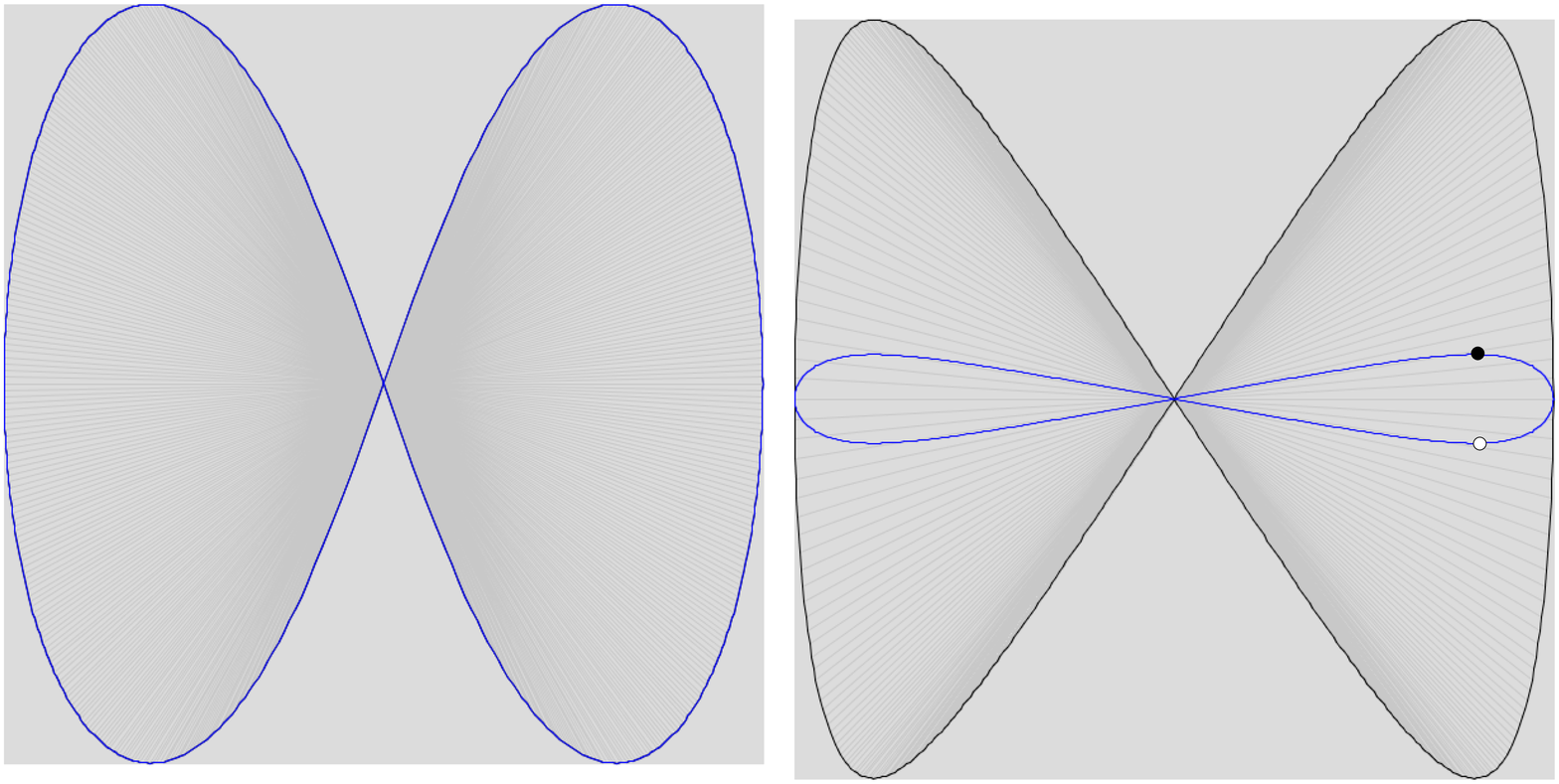}}
\newline
    {\bf Figure 1:\/} A hint of the bowtie.
  \end{center}

  Now we sketch the proof of the Bowtie Theorem.
  Our argument fleshes out the
      outline proposed in \cite{M}.
  Let $A(t)$ be the area of the region -- i.e., the two lobes --
  bounded by $C(t)$.
Let $X(t)$ and $Y(t)$ be such that
 $[-X(t),X(t)] \times [-Y(t),Y(t)]$ is the
bounding box of $C(t)$.
  The main thrust of our proof is establishing the following three
  formulas.
  
\begin{equation}
\label{AS0}
\lim_{t \to T} \frac{A(t)}{T-t}=2\pi, \hskip 14 pt
\liminf_{t \to T} \kappa(\pi/2,t) Y(t) \geq \pi/2, \hskip 14 pt
\liminf_{t \to T} \frac{X(t)}{(T-t)\kappa(\pi/2,t)} \geq 2
\end{equation}
The first of these formulas is essentially the same as the bound in
\cite{G3}, but sharpened by the fact, which we prove, that the angle
at the double point tends to $0$ as $t \to T$.
The second formula is an easy application of the Grim Reaper
Theorem.
The third formula follows from a well-chosen rescaling argument combined
with the Sturmian Principle.
These formulas combine to give the upper bounds
\begin{equation}
\label{AS3}
\limsup_{t \to T} \frac{A(t)}{X(t)Y(t)} \leq 2, \hskip 30 pt
\limsup_{t \to T} {\rm area\/}(L_t(C(t))) \leq 2
\end{equation}
The first bound immediately implies the second.
On the other hand, it follows from convexity that
$L_t(C(t))$ has area at least $2$.  We conclude that
\begin{equation}
  \label{AS5}
  \lim_{t \to T} {\rm area\/}(L_t(C(t))) = 2, \hskip 14 pt
  \lim_{t \to T} \kappa(\pi/2,t) Y(t) =\pi/2, \hskip 14 pt
\lim_{t \to T} \frac{X(t)}{(T-t)\kappa(\pi/2,t)} = 2
\end{equation}

Similar asymptotic results are obtained for everywhere locally convex curves in \cite{ANG4} and \cite{FM}.
We use the middle equation in Equation \ref{AS5}, the $Y$ bound, to
prove:

\begin{lemma}[Migration]
  The point on $L_t(C(t))$ having positive first coordinate
  and largest second coordinate converges to $(1,1)$ as $t \to T$.
\end{lemma}

Now, suppose that there is a sequence
$\{t_n\}$ for which
     $L_{t_n}(C(t_n))$ does not converge in the
     Hausdorff topology to the bowtie.  We
      pass to a further subsequence so that
      $\{L_{t_n}(C(t_n))\}$ has some limit point $(x,y)$
      in the positive quadrant that lies
        outside the region bounded by the bowtie.
        Let $\Psi$ be the polygonal figure $8$, with
        $4$-fold dihedral symmetry, whose
        right lobe is the convex hull of points
        $$(0,0), \hskip 30 pt (1,\pm 1), \hskip 30 pt (x,\pm y)$$
        $\Psi$
        has area greater than $2$.  By
        symmetry and convexity, the region bounded by
        $L_{t_n}(A(t_n))$ contains a subset which
        converges to
        $\Psi$ in the Hausdorff metric.
       But then the area of
      $L_{t_n}(A(t_n))$ cannot converge to $2$.  This
      contradicts Equation \ref{AS5}.
      \newline

      In \S 2 we prove that the flow preserves monotonicity.  In \S 3
      we establish some basic asymptotic facts about $C(t)$ as $t \to
      T$, such as the decay of the angle at the double point and the
      area estimate. Some readers might want to know that \S 3
      contains
      all the results where we explicitly use the monotonicity of the
      curvature. In \S 4 we prove the Grim Reaper Theorem.  In
      \S 5 we establish the second and third formulas in Equation
      \ref{AS0}.
      In \S 6 we prove the Migration Lemma, thereby completing
      the proof of the Bowtie Theorem.
      \newline
      \newline
  \noindent
{\bf Acknowledgements:\/}
We thank Peter Doyle and
Mike Gage for some helpful conversations.
We thank Brown University and
 the National Science Foundation for their support.
 The first author thanks Princeton University for its support.
The second author thanks
the Simons Foundation, in the form
of a Simons Sabbatical
Fellowship, and the Institute for Advanced Study
for a year-long membership funded
by a grant from the Ambrose Monell Foundation. We would also like to thank the anonymous referees for their insightful comments.

\newpage

\section{Preservation of Monotonicity}
\label{GR}

\subsection{Strictly Parabolic Equations}

In this chapter we prove that
the curve shortening flow preserves the monotonicity condition.
We begin with a discussion of strictly parabolic equations
and two of their basic properties.
We follow the
notation in \cite{EV} and \cite{ANG}.

Let $U$ be an open interval containing $[x_0,x_1]$.
We suppose that $u: U \times [0,\tau]$ satisfies the equation
\begin{equation}\label{GPDE}
u_t = a(x,t) u_{xx} +b(x,t)u_x + c(x,t) u.\end{equation}
This equation is called
{\it strictly parabolic\/} if and only if 
$a(x,t)$, $b(x,t)$, and $c(x,t)$ are smooth and
$a(x,t)>0$.  We assume that $u$ satisfies a
strictly parabolic PDE.
Here is the well-known {\it Maximum Principle\/}.

\begin{theorem}[The Maximum Principle]
  Suppose that $u \not = 0$ on
  $\{x_0,x_1\} \times [0,\tau]$ and
  also nonzero on $[x_0,x_1] \times \{0\}$.
  Then $u$ is nonzero on
  $[x_0,x_1] \times [0,\tau]$.
\end{theorem}

Geometrically we are looking at the behavior of
$u$ on a rectangle.  If we know that $u$ is nonzero
on $3$ sides of $\partial R$ then we know
$u$ is nonzero on all of $R$.   
The side $[x_0,x_1] \times \{0\}$ is
the bottom side of $R$  and the side $[x_0,x_1] \times \{\tau\}$ is
the top.  Here we are picturing time as running vertically
and space as running horizontally.

Here is the well-known Sturmian Principle.

\begin{theorem}[The Sturmian Principle]
  Suppose $u$ is nonzero on
  $\{x_0,x_1\} \times [0,\tau]$.
  Then the number $N_t$ of times $u(*,t)$ vanishes on $(x_0,x_1)$
is non-increasing with time. Moreover, if
$u(*,t)$ vanishes to second order somewhere on $(x_0,x_1)$ then
$N_{t'}<N_t$ for all $t' \in (t,\tau]$.
\end{theorem} 

C. Sturm discovered this principle in 1836. see
\cite{S}.
The proof of the above version of the Sturmian Principle may be found in \cite{ANG2}.
For more references about these theorems, see \cite{EV} or
\cite{ANG2}.
Note that if $u,v$ solve the same strictly parabolic equation then so
does $w=u-v$.   This yields the following corollary.

\begin{corollary}
  \label{STURM1}
  Suppose $w$ is nonzero on
  $\{x_0,x_1\} \times [0,\tau]$.
  Then the number $N_t$ of zeroes for $w(*,t)$ on $(x_1,x_2)$
  is finite and non-increasing. 
  Moreover, at any time $t$ when $w(*,t)$ vanishes to second order,
  we have $N_{t'}<N_t$ for all $t'>(t,\tau]$.
\end{corollary}

\noindent
{\bf Curvilinear Domains:\/}
Rectangular domains are too
restrictive for one of our purposes.   The same
principles work when the rectangle in
question is replaced by a piecewise
analytic quadrilateral $\cal Q$  with the
following two properties:
\begin{enumerate}
\item The top and bottom sides are line segments, with the bottom
  one corresponding to time $0$ and the top one corresponding to time $\tau$.
\item The function $w$ does not vanish on the other two sides.
\end{enumerate}
The other two sides play the role of
 $\{x_0\} \times [0,\tau]$ and
$\{x_1\} \times [0,\tau]$.
The main issue is that the non-horizontal sides prevent
zeros from ``leaking in or out''.

\begin{center}
\resizebox{!}{1.6in}{\includegraphics{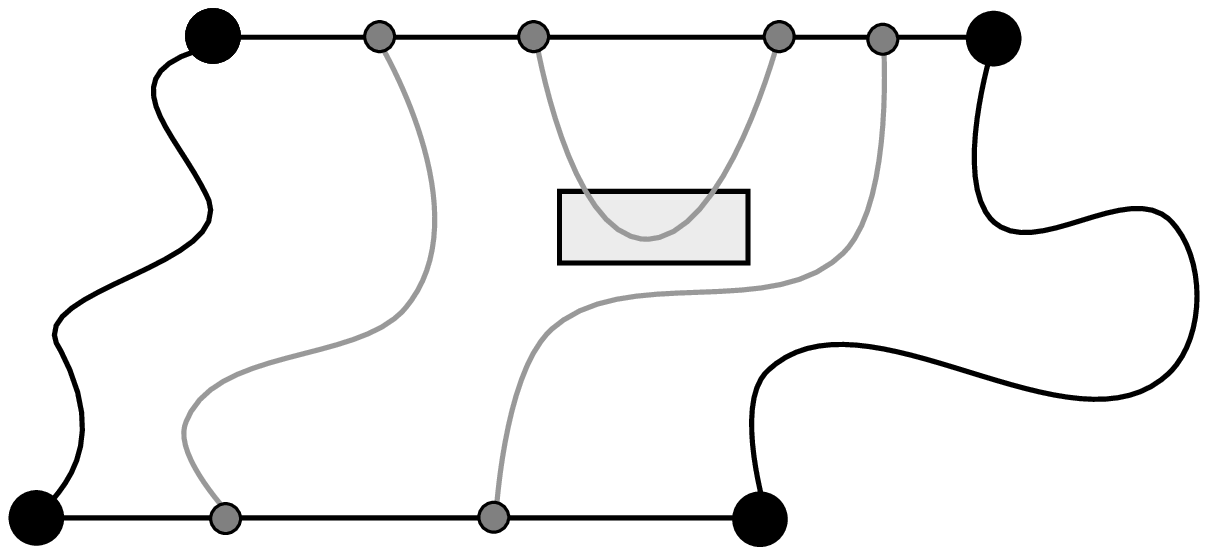}}
\newline
    {\bf Figure 2:\/} The Curvilinear case
\end{center}

Let us explain why the rectilinear principle implies
the curvilinear principle.  Suppose we have a situation where
$w$ has $m$ zeros on the bottom of $\cal Q$ and $n>m$ on the top of $\cal Q$.
Let $I$ be the set of times where $w$ has more than $m$
zeros.  Let $t=\inf I$.  The zeros of $w$ at times
converging to $t$ cannot converge to the non-horizontal
sides of the domain.  Hence at least two of them must
coalesce.  But then we can find a small rectangle
$R \subset \cal Q$ which surrounds these coalescing points.
See the small shaded rectangle in Figure 2.
(If more points coalesce, the picture would look more complicated.)
This gives a contradiction to the rectilinear
principle. 
  
\subsection{Evolution Equations}

The evolution equation for $\kappa$ is given in Equation \ref{MAIN}.
Here it is again.
\begin{equation}
  \label{SP1}
  \kappa_t=(\kappa^2)\kappa_{\theta \theta} + (0)\kappa_\theta +(\kappa^2)\kappa.
\end{equation}
This equation is strictly parabolic.

Let $u=\kappa_{\theta}$.
Differentiating Equation \ref{MAIN} with respect to $\theta$ we get
the evolution equation for $u$:
\begin{equation}
\label{SP2}
u_t=(\kappa^2)u_{\theta \theta} + (2 \kappa u)u_{\theta} +
(3 \kappa^2)u
\end{equation}
All we need to know about this equation is that it is strictly
parabolic.  Also, we only need this equation in this chapter.

Let $y(x,t)$ be the evolution of the height of the curve $C(t)$. 
Let
\begin{equation}
  k(x,t):= k(x,y(x,t))
\end{equation}
denote the signed
curvature at the point $(x,y(x,t))$.
For fixed $t$, the curve is defined in terms of $x$ and $y$ and the
curvature
is given in terms of partial derivatives of $y$ with respect to $x$ holding $t$ fixed.
Note that the domain for $x$ is shrinking to a point.
The following evolution equation for $\mu=k_x$ is derived in \cite{G}:
\begin{equation}
\label{SP4}
\mu_t =(\zeta) \mu_{xx} 
+(-2 y_x y_{xx} \zeta^2) \mu_x +(3k^2)\mu, \hskip 30
pt
\zeta=\frac{1}{1+y_x^2}.
\end{equation}
Again, all we need to know about this equation is that it is strictly
parabolic
and we only need this equation in this chapter.

Equations \ref{SP1} and \ref{SP2} are valid on the domain
\begin{equation}
{\cal D\/}=\bigcup_{t \in [0,T)} (-\alpha(t),\pi+\alpha(t)) \times \{t\}.
\end{equation}
Equation \ref{SP4} is valid away from places where our curve
has vertical tangents. In particular on any
time range $[0,t]$ for $t<T$ it is valid on
each strand in some fixed neighborhood 
of the double point.

\subsection{Monotonicity}
\label{preserve}

Now we prove that the curve shortening flow preserves
the monotonicity property.  We assume that $C(0)$ is monotone.

\begin{lemma}
  \label{forward1}
  If $\kappa_{\theta}(\theta,t)>0$ for all
  $\theta \in (0,\pi/2)$ and all $t \leq t_0$, then
  \begin{enumerate}
    \item
      $k_x(0,t)>0$ for all $0\leq t \leq t_0$.
      \item  $\kappa_{\theta \theta}(\pi/2,t) \not =0$ for all $0 \leq
        t \leq
        t_0$.
        \end{enumerate}
  \end{lemma}

  \begin{proof}
    For the first statement, it
    suffices to prove that $k_x(0,t_0)>0$.
    Not first that
    the last monotonicity property implies that $k_x(0,0)>0$.
The reason is that near the double point the $x$-coordinate
is a smooth invertible function of arc length.
We apply the Maximum principle to a
  rectangle of the form $[-\epsilon,\epsilon] \times [0,t_0]$
  and we get a contradiction.

  For the second statement,
  it suffices to prove that $\kappa_{\theta \theta}(\pi/2,t_0) \not =0$.
This is an application of the Sturmian Principle for $\kappa_{\theta}$
with respect to a
rectangle of the form $[\pi/2-\epsilon,\pi/2+\epsilon] \times
[0,t_0]$.  We are assuming that $\kappa_{\theta \theta}(\pi/2,0) \not
=0$.
Hence $\kappa_{\theta}(*,0)$ only vanishes to first order at $\pi/2$. Also
$\kappa_{\theta}$ does not vanish on
the vertical sides of the rectangle, by symmetry.
\end{proof}

\begin{lemma}
  \label{forward2}
   $\kappa_{\theta}(*,t)>0$ on $(-\alpha(t),\pi/2)$ for all $t<T$.
\end{lemma}

\begin{proof}
  Recall that $u=\kappa_{\theta}$. We need
  to show is that $u(\theta,t)>0$ on the domain $\cal D$.
  Suppose this fails.
  Let $I$ denote the set of times $t'$ for which $u(*,t')$ vanishes
  somewhere.
  Let $t_0=\inf I$. There are several cases to consider.

Suppose first that $t_0 \in I$.
Then there is some $(\theta,t_0) \in \cal D$ such that
$u(\theta,t_0)=0$ but $u(*,t)>0$ for all $t \in [0,t_0)$.
In this case we get a contradiction by applying the
Maximum Principle to $u$ on a rectangle
$[\theta-\epsilon,\theta+\epsilon] \times [t_0-\epsilon,t_0].$
For sufficiently small $\epsilon$ this rectangle belongs to
$\cal D$.  Since  $u$ is analytic we can further
choose $\epsilon$ so that $u(\theta \pm \epsilon,t_0)>0$.
We now contradict the Maximum Principle. Hence
$t_0 \not \in I$.

Let $(\theta_n,t_n)$ be a sequence of points in $\cal D$
such that $u(\theta_n,t_n) = 0$ and $t_n \to t_0$.
Since $t_0 \not \in I$, we must have (after using symmetry and
passing to a subsequence) either
$\theta_n \to -\alpha(t)$ or $\theta_n \to \pi/2$.
By Lemma \ref{forward1} we know that
$k_x(0,t_0)>0$ and $\kappa_{\theta \theta}(\pi/2,t_0) \not =0$.
Consider the cases.

\begin{itemize}

\item
Suppose $\theta_n \to -\alpha(t_0)$.  By the Chain rule,
$k_x(x_n,t_n)=0$ for a sequence $x_n \to 0$.
But then $k_x(0,t_0)=0$ by continuity.  This is a contradiction.

\item
Suppose $\theta_n \to \pi/2$.  Since we are now in the
interior of the domain $\cal D$ and $u$ is a smooth function, we have
$$\kappa_{\theta \theta}(\pi/2,t_0)=\lim_{n \to \infty}
\frac{u(\pi/2,t_n)-u(\theta_n,t_n)}{\pi/2-\theta_n}=0.$$
This is a contradiction.
\end{itemize}
This completes the proof.
\end{proof}

Let us now check that $C(t)$ is monotone for any $t<T$.
\begin{itemize}
  \item The curve shortening flow preserves analyticity, hence $C(t)$
    is analytic.
    \item The curve shortening flow respects symmetry, so $C(t)$ has
      $4$-fold symmetry.
    \item As we mentioned in the introduction, Angenent proves in
      \cite{ANG2} and \cite{ANG3} that
        $C(t)$ has convex lobes for all $t<T$.
      \item Lemma \ref{forward2} says exactly that
          $\kappa_{\theta}(\theta,t)>0$ for $\theta \in
          (-\alpha(t),\pi/2)$.
          \item Lemma \ref{forward1} shows that $\kappa_{\theta
              \theta}(\pi/2,t) \not =0$.
          \item Lemma \ref{forward1} shows that the signed curvature of
            $C(t)$, as a function of arc length, does not vanish to second order
            at the double point.
\end{itemize}

\newpage

\section{Some Asymptotic Results}
\label{usecurv}

The results in this chapter use the assumption
that $C(0)$  is monotone.  As we proved in the last chapter,
this means that $C(t)$ is monotone for all $t<T$.
One tool we will use several times is the
well-known Tait-Kneser Theorem from differential geometry.
One can find a proof in practically any book of
differential geometry.

 \begin{theorem}[Tait-Kneser]
    Suppose $\gamma$ is a curve of strictly monotone increasing
    curvature.  Then the osculating disks of $\gamma$
    are strictly nested.  The largest one is at the
    initial endpoint and the smallest one is at the final
    endpoint.  In particular, $\gamma$ lies inside the osculating
    disk at its initial endpoint and outside the osculating disk at
    its final endpoint.
  \end{theorem}

  Recall that $[-X(t),X(t)] \times [-Y(t),Y(t)]$ is the
bounding box of $C(t)$.

  \begin{lemma}[Bounding Box]
    \label{bb0}
    $\lim_{t \to T} Y(t)/X(t) =0$.
  \end{lemma}

  \begin{proof}
    The perimeter of $C(t)$ and
    the area of the region bounded by $C(t)$ are respectively within a
    factor of $2$ of the perimeter and area of the bounding box of $C(t)$.
    Thus, Grayson's isoperimetric result tells us that the aspect ratio
    of the bounding box tends to $0$.  This means that either
    $Y(t)/X(t) \to 0$ or $Y(t)/X(t)\to \infty$ as $t \to T$.

Given that  $X_t(t)=-\kappa(\pi/2,t)$ and 
$Y_t(t)=-\kappa(0,t)$, we have
$$Y(t)=\int_t^T \kappa(0,u)\ du<
\int_t^T \kappa(\pi/2,u)\ du = X(t).$$
The inequality uses the monotonicity condition.
This rules out the second option.
\end{proof}

\begin{lemma}[Curvature Blowup]
  \label{blow00}
  $\lim_{t \to T} \kappa(\theta,t)=\infty$ for any $\theta \in (0,\pi/2]$.
\end{lemma}

\begin{proof}
  Let $\Gamma(t)=C(t)/X(t)$.
  This is a rescaled version of $C(t)$ whose
bounding box has width $2$.  The height
of the bounding box tends to $0$ by
Lemma \ref{bb0}.
Let $$K(\theta,t)=X(t)\kappa(\theta,t)$$ be
the curvature of
$\Gamma(t)$ at the point where the tangent
angle is $\theta$.    Since we have
$\lim_{t \to T} X(t)=0$, it suffices to prove (say)
that $K(\theta,t) \geq 2\sin(\theta)$.

Let $\Delta=\Delta(\theta,t)$ be the osculating disk to $\Gamma(t)$ at
$\Gamma(\theta,t)$.  Since the curvature of
$\Gamma(t)$ is monotone increasing,
the Tait Kneser Theorem says that the arc of
$\Gamma(t)$ connecting the origin to $\Gamma(\theta,t)$ lies
outside $\Delta$.  This forces $\partial \Delta$
to cross the horizontal line $L$ through $\Gamma(\theta,t)$
twice
inside the bounding box and in the positive quadrant.
The intersection $L \cap \Delta$ has
length at most $1$ and the angle between $L$ and
$\partial \Delta$ at the intersection points is $\theta$.
It follows from trigonometry that $\Delta$ has
radius at most $1/(2\sin\theta)$.  Hence $K(\theta,t) \geq 2
\sin(\theta)$.
\end{proof}

\begin{lemma}[Angle Decay]
  \label{angle00}
  $\lim_{t \to T} \alpha(t)=0$.
\end{lemma}

\begin{proof}
  Let $\Gamma(t)$ be as in the previous lemma.
Suppose that there is a sequence of
times $t_n \to T$ such that
$\alpha(t_n)>\delta>0$ for some
constant $\delta$.   Let
$L$ be the line through the origin
which makes an angle of $\delta/2$ with
the $x$-axis.  Again, the
height of the bounding box for $\Gamma(t_n)$ tends to
$0$ as $n \to \infty$.  Hence, $L$
hits the top of the bounding box at a point
whose distance to the origin tends to $0$
as $n\to \infty$.

By construction $\Gamma(t_n)$ starts out
from the origin lying to the left of $L$.
Since $\Gamma(t_n)$ lies inside its bounding box,
we see that 
$\Gamma(t_n)$ crosses $L$ at some point
$p_n$ such that $\|p_n\| \to 0$.
The total variation of the tangent angle of $\Gamma(t)$ along the arc
connecting $(0,0)$ to $p_n$ is, by convexity, at least $\delta/2$.
Since the length of this arc tends to $0$, and since
the curvature is monotone increasing, the curvature of
$\Gamma_n$ at $p_n$ is eventually at least $4$.

By the Tait-Kneser
Theorem the arc of $\Gamma(t)$ connecting $p_n$
to $(1,0)$ is trapped in a disk of radius
$1/4$ which contains $p_n$ in its boundary.
This is a contradiction.
\end{proof}

\begin{corollary}[Area Asymptotics]
  \label{AA}
The first formula in
Equation \ref{AS0} is true.
\end{corollary}

\begin{proof}
  Let $k(s,t)$ denote absolute value of the
  curvature as a function of arc length and time.
  Consider the two curves $C(t)$ and $C(t+\delta)$ for some
  very small $\delta$. At any given point $(s,t)$ on $C(t)$ the
  distance from $C(t)$ to $C(t+\delta)$ equals $\kappa(s,t) \delta$.
  up to order $(\delta)^2$.  So, up to order $\delta^2$ the
  total change in area is
  $$\int k(s,t) ds=2\pi + 2 \alpha(t).$$
  By the definition of the derivative, we therefore have
  $A_t(t) = -2\pi -2 \alpha(t)$.
    Hence
  $$\lim_{t \to T} A_t(t)=-2\pi-\lim_{t \to T} \alpha(t)=-2\pi.$$
  We set $B(t)=T-t$.
  Since $\lim_{t \to T} A(t)=\lim_{t \to T} B(t)=0$ we have
  $$\lim_{t \to T} \frac{A(t)}{T-t}=
  \lim_{t \to T} \frac{A(t)}{B(t)}=
  \lim_{t \to T} \frac{A_t(t)}{B_t(t)}=\frac{-2\pi}{-1}=2\pi$$
  by L'H\^{o}pital's rule.
\end{proof}

Our final estimate involves a different rescaling of our curve,
and it will be useful when we establish the last formula in
Equation \ref{AS0}.  The name of the lemma will also
become clear later on. (The reader might want to just skim this
result on the first pass.) Let
\begin{equation}
  D(t)=\frac{1}{\sqrt{T-t}} C(t).
\end{equation}
Given Corollary \ref{AA}, the area of $D(t)$ converges to $2\pi$ as
$t \to T$.  This curve is getting very long and thin.
Let $D_1(t)$ denote the arc of $D(t)$ which starts at the origin,
has length $1$, and starts out moving into the positive quadrant.
If $t$ is sufficiently close to $T$, the arc $D_1(t)$ lies entirely
in the positive quadrant. 

Let $K(\theta,t)=\kappa(\theta,t) \sqrt{T-t}$ denote the
curvature of $D(t)$ at $D(\theta,t)$.
Let ${\bf n\/}(\theta)$ denote the outward normal to $D(t)$ at
$D(\theta,t)$.  Even though we picture this vector as based at
$D(\theta,t)$, it is independent of $t$.  Finally define
\begin{equation}
  P(\theta,t)=D(\theta,t) \cdot {\bf n\/}(\theta).
  \end{equation}

\begin{lemma}[Nodal Function Estimate]
  \label{nega}
  If $D_1(t)$ lies in the positive quadrant then
  $P(\theta,t)<K(\theta,t)$ for any angle $\theta$
  such that $D(\theta,t) \in D_1(t)$.
\end{lemma}

\begin{proof}
  We fix $t$ and suppress it from our notation.
  Let $s$ denote the arc length along $D$ chosen
  so that $s=0$ corresponds to the origin.
  Let $T_s$ denote the unit tangent vector to
  $D$ at $s$, chosen so that the first coordinate is positive.
  
  Consider the situation at a point corresponding to
  $s<1$ along $D$.  Let $a$ denote the angle between
  $T_s$ and $D(s)$.
  Note that $P>0$ by convexity. We have
  $$P(s) = D(s) \cdot {\bf n\/}(s) = \|D(s)\| \sin(a)<
  s \sin(a)<sa,$$

\begin{center}
 \resizebox{!}{1.8in}{\includegraphics{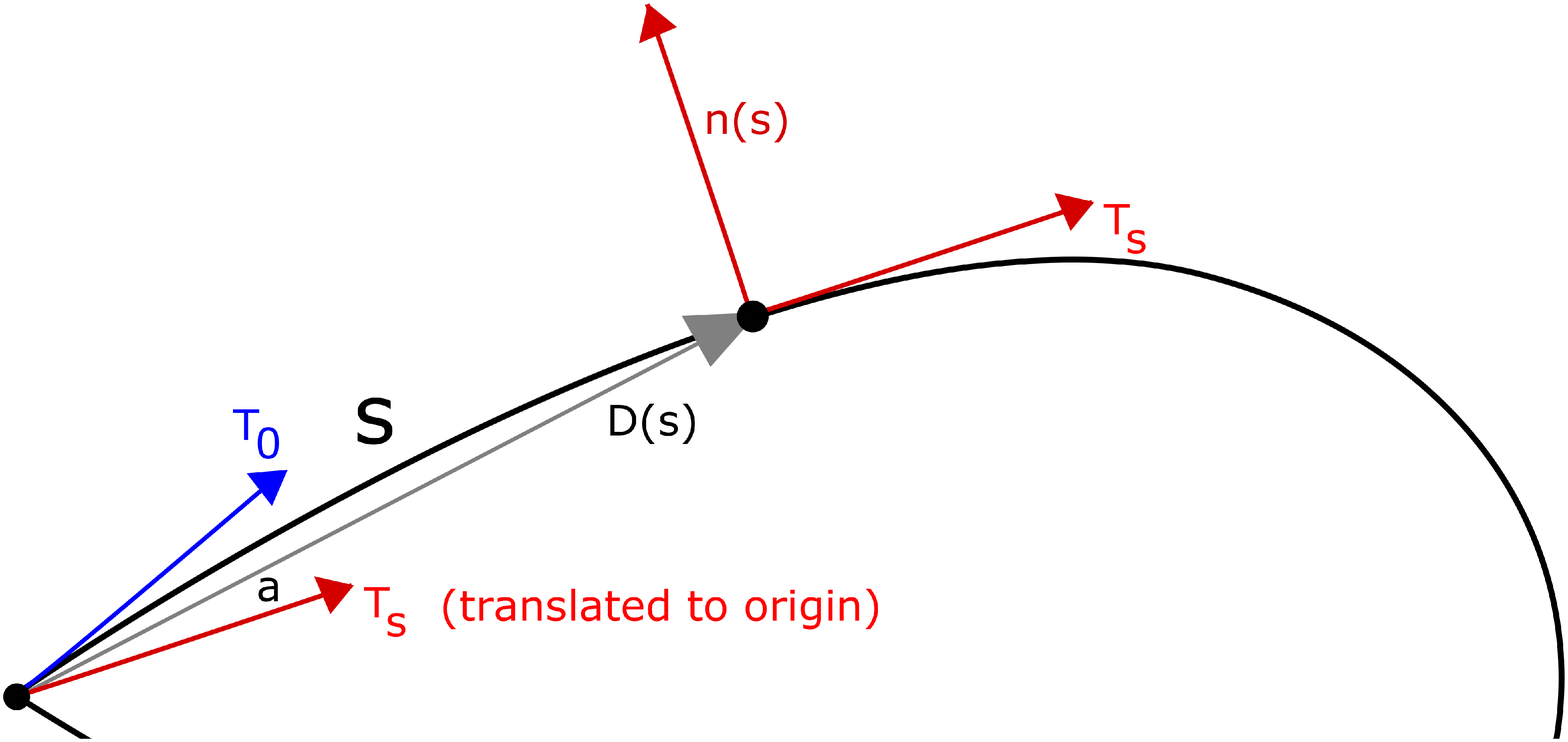}}
\newline
{\bf Figure 3:\/} The relevant vectors
\end{center}

  By convexity, the vector $D(s)$ lies in the sector bounded
  by $T_0$ and $T_s$.  Hence $a$ is less than the
  angle $\overline a$ between $T_0$ and $T_s$.
  But then, given the definition of curvature,
    $$a <\overline a =\int_0^s K(\sigma) d\sigma < sK(s).$$
    The last inequality uses the monotonicity of the curvature.
  Putting our two estimates together we have
  $P(s)<s^2 K(s)<K(s)$.
\end{proof}

\section{The Grim Reaper Theorem}
\label{REAPER}

\subsection{Counting Zeros}

Our first lemma has nothing to do with the
flow.  A very similar principle is used in
\cite{ANG}.
Let $J \subset {\bf R\/}$ be some interval.
Call a function $g: J \to {\bf R\/}$ {\it small\/}
if
\begin{equation}
  \sup_{J} g^2 + (g')^2<1.
\end{equation}
Call $J$ {\it small\/} if it has length
at most $\pi$.  Every small interval is contained
in a closed interval of length $\pi$.  Closed intervals
of length $\pi$ count as being small.

\begin{lemma}
  \label{scaling}
  If $g$ is a small function and $J$ is a small interval
  then the difference $w(x)=g(x)-\sin(x)$ vanishes at most twice on $J$,
  counting multiplicity.
\end{lemma}

\begin{proof}
  Let $f(x)=\sin(x)$.  We note the
crucial property that
$$f^2+(f')^2=1>g^2 + (g')^2.$$
  Let $F$ and $G$ respectively denote the
  graphs of $F$ and $G$.  These graphs
  must be transverse wherever they
  intersect.  Otherwise we would have
$g^2 + (g')^2 =1$ at an intersection
point. This is impossible.  We show that
$f=g$ at most twice.  Given the transversality
just mentioned, this is equivalent to the
statement that $w=g-f$ vanishes at most
twice on $J$, counting multiplicity. 

As usual in calculus, say that $x \in J$ is an
{\it extreme point\/} if $f'(x)=0$.
The only way that $J$ can contain two extreme
points is if $J$ has length $\pi$, and
the endpoints are the two extreme points,
and $|f|=1$ at these endpoints.
In this case $f \not = g$ at the endpoints because
$|g|<1$.  So, even in this case, 
we can replace $J$ by a
smaller interval which contains all the
points where $f=g$.  Thus, we can assume
without loss of generality that
$J$ contains at most one extreme point.

Suppose first that $J$ has no extreme points.
Then $f$ is either monotone increasing on
$J$ or monotone decreasing.  Consider the
case when $f$ is monotone increasing.
Suppose it happens that there are two consecutive points
$x_1,x_2 \in J$ where $f$ and $g$ agree.
The portion of $G$ between $(x_1,g(x_1))$ and
$(x_2,g(x_2))$ either
lies above $F$ or below.  In the first case we
have $g'(x_1)>f'(x_1)$, which is a contradiction.
In the second case we have
$g'(x_2)>f'(x_2)$ and we have the same
contradiction.  Hence $f(x)=g(x)$ for
at most one point $x \in J$.  The
same argument works when $f$ is monotone
decreasing on $J$.

Now consider the case when $J$ has exactly
one extreme point.  In this case we can
write $J=J_1 \cup J_2$ where $f$ is monotone
on each $J_i$.  In this case, the same argument
above, applied to each of these sub-intervals,
shows that they each have at most one point
where $f=g$.  Hence $J$ has at most $2$ such points.
\end{proof}

\subsection{The Sine Lemma}

Here is the crucial step in the proof of the Grim Reaper Theorem.
This section is devoted to proving the following result.

\begin{lemma}[Sine]
  \label{EM}
  Let $J$ be any closed interval
  contained in $(0,\pi)$.
  Let  $\epsilon>0$ are given.
  If $t$ is sufficiently close to $T$ then
  $$\bigg|\frac{\kappa_{\theta}(\theta,t)}{\kappa(\theta,t)}-
  \frac{\cos(\theta)}{\sin(\theta)}\bigg|< \epsilon,$$
  for all $\theta \in J$.
\end{lemma}

 We will assume for
the sake of contradiction that
there is a sequence of times $t_n \to T$ and a sequence
$\{\theta_n\} \in J$ such that
\begin{equation}
\label{contra}
  \bigg|\frac{\kappa_{\theta}(\theta_n,t_n)}{\kappa(\theta_n,t_n)}-
\frac{\cos(\theta_n)}{\sin(\theta_n)}\bigg|>\epsilon.
\end{equation}
Passing to a subsequence we can assume that $\theta_n \to \theta_0 \in
J$.
By compactness of $J$ we can choose a constant
$\Sigma=\Sigma(J,\epsilon)>0$ so that 
\begin{equation}
\label{SHIFT}
\bigg|\frac{\cos(\phi+\theta_n)}{\sin(\phi+\theta_n)}-
\frac{\cos(\theta_n)}{\sin(\theta_n)}\bigg|>\epsilon \hskip 20 pt
\Longrightarrow \hskip 20 pt |\phi|> \Sigma,
\end{equation}
as long as $\phi+\theta_n \in (0,\pi)$.

Call the non-horizontal sides of our domains the
{\it sidewalls\/}.
Thanks to Lemma \ref{angle00} we can omit the initial
portion of our evolution and arrange that
\begin{equation}
  \sup_{t \in [0,T)} \alpha(t) < 10^{-100} \Sigma.
\end{equation}
We are making the horizontal displacement of
the sidewalls of $\cal D$ extremely small
in comparison to the other relevant quantities that arise below.
We don't need the factor of $10^{-100}$; we add it for emphasis.

Let
\begin{equation}
  C=\sup_{\theta \in [0,\pi]} \kappa^2(\theta,0) + \kappa_{\theta}^2(\theta,0), \hskip 30 pt
  B_n=\kappa^2(\theta_n,t_n)+\kappa_\theta^2(\theta_n,t_n).
\end{equation}
By Lemma \ref{blow00} there is some $n$ such that $B_n>C$.
Our motivation for taking $B_n>C$ is the following
corollary of Lemma \ref{scaling}.

\begin{corollary}
\label{scaling2}
Suppose
$$\sup_{\theta \in J} \kappa ^2(\theta,0) + \kappa_{\theta}^2(\theta,0) \leq C.$$
Let $S(\theta)=\sqrt B\sin(\phi + \theta)$ for any value $\phi$.
If $B>C$ then $w(*)=\kappa(*,0)-S(*)$ vanishes at most twice
on $J$, counting multiplicity.
\end{corollary}

\begin{proof}
This follows from Lemma \ref{scaling} by
symmetry and scaling. 
\end{proof}

We fix $n$ for which $B_n>C$. We set $B=B_n$ and $t=t_n$.
There is some angle $\phi$ such that
$$\frac{\kappa_{\theta}(\theta_n,t)}{\kappa(\theta_n,t)}=
\frac{\cos(\phi+\theta_n)}{\sin(\phi+\theta_n)}.$$
For this choice of $\phi$ we have
\begin{equation}
\label{MATCH}
  S(\theta_n)=\sqrt B\sin(\phi+\theta_n)=\kappa(\theta_n,t), \hskip 20 pt
  S_{\theta}(\theta_n)=\sqrt{B}\cos(\phi + \theta_n) =\kappa_{\theta}(\theta_n,t).
\end{equation}
Our function $S$ determines a unique interval $I$ of length $\pi$ such that
$S>0$ on the interior of $I$ and $\theta_n \in I$.  Note also that $S=0$ on $\partial I$.
Let $\Omega=I \times [0,t]$.  This is exactly the
domain considered in \cite{ANG}, but now our proof departs from
\cite{ANG}.  

\begin{lemma}
One sidewall of $\Omega$ is disjoint from the closure of $\cal D$
and the other sidewall of $\Omega$ lies in $\cal D$.
\end{lemma}

\begin{proof}
The properties of $S$ imply the following:
\begin{equation}
\bigg|\frac{\cos(\phi+\theta_0)}{\sin(\phi+\theta_0)}-
\frac{\cos(\theta_0)}{\sin(\theta_0)}\bigg| =
\bigg|\frac{\kappa_{\theta}(\theta_0,t)}{\kappa(\theta_0,t)}-
\frac{\cos(\theta_0)}{\sin(\theta_0)}\bigg|>\epsilon.
\end{equation}
By equation \ref{SHIFT}, we have $|\phi|>\Sigma$.

If we had $\phi=0$ we would have $I=[0,\pi]$.
As it is, we have
$|\phi|>\Sigma$.  This shifts $I$ and
$\Omega$ by at least $\Sigma$ to the left or to the right.
Given our bound on the horizontal displacement of
the sidewalls of $\cal D$,
this shift causes one sidewall or the other to stick out completely.
See Figure 4 below.  At least one point of $I$ lies in $(0,\pi)$ and
the total width of $I$ is $\pi$.  Hence $I$ cannot both contain
points less than $0$ and greater than $\pi$.  This means that
the other sidewall lies inside $\cal D$.
\end{proof}

We now create a new domain $\cal Q$  by intersecting
$\Omega$ with $\cal D$ and pushing in the
curvilinear sidewall a bit.  We treat
the case when $\Omega$ sticks out on the
left.  The other case is essentially the same.

\begin{center}
\resizebox{!}{1.7in}{\includegraphics{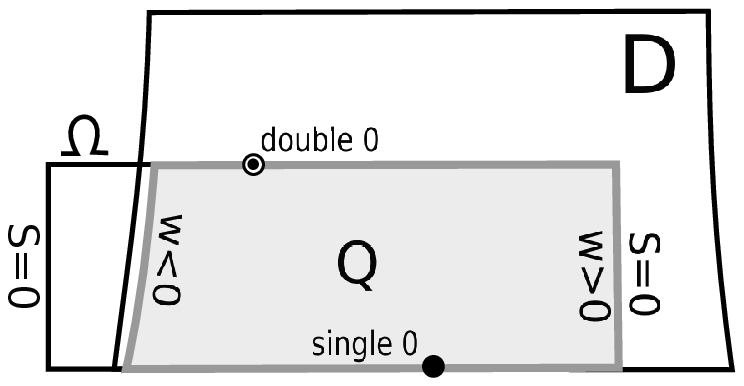}}
\newline
    {\bf Figure 4:\/} The new domain $\cal Q$, shaded.
\end{center}

Define
\begin{equation}
w(\theta,t) = \kappa(\theta,t)-S(\theta).
\end{equation}
The function $S$ is a stationary solution
to Equation \ref{MAIN}, meaning that $S_t=0$.
This means that $w$ is exactly the sort of
difference of solutions to which the
Sturmian Principle applies.
Let us examine the behavior of $w$ on the boundary of
$\cal Q$.
\newline
\newline
{\bf Left:\/} Since
$\kappa$ limits to $0$ on the sidewalls of $\cal D$
and $S>0$ on the left sidewall of $\cal D$, we can
by compactness
make the perturbation small enough so that $w<0$ on
the left sidewall of $\cal Q$.
\newline
\newline
{\bf Right:\/}
The right sidewall of $\cal Q$ lies in $\cal D$.
Since $S=0$ on the right sidewall  of $\cal Q$ and
$\kappa>0$ everywhere in $\cal D$, we have
$w>0$ on the right sidewall of $\cal Q$.
\newline
\newline
    {\bf Bottom:\/}
    Applying Corollary \ref{scaling2} to the
bottom side $J$ of $\cal Q$,
we see that $w(*,0)$ vanishes at most twice on $J$
counting multiplicity.
Since $w$ has opposite signs on
the sidewalls of $\cal Q$ the number of
zeros of $w$ on $J$ 
is odd, counting multiplicity. Since this
number is at most $2$, it must be exactly $1$.
In short, $w$ vanishes exactly once on
the bottom side of $\cal Q$, counting multiplicity.
\newline
\newline
    {\bf Top:\/}
On the top side $J'$ of $\cal Q$ we have arranged
that $w$ and $w_{\theta}$ vanish at $(\theta_0,t)$.
This means that $w$ vanishes at least twice,
counting multiplicity, on $J'$.
We have shown this double point in Figure 4.
Since $w$ has opposite signs on
the sidewalls of $\cal Q$ the number of
zeros of $w$ on $J'$
is odd, counting multiplicity. Since
this number is at least $2$ it is actually
at least $3$. In short,
$w$ vanishes at least $3$ times on the top side
of $\cal Q$ counting multiplicity.
\newline

The above properties violate the Sturmian
Principle for
(Equation \ref{MAIN}, $\cal Q$, $w$).
This completes the proof of the
Sine Lemma.

\subsection{The End of the Proof}

In this section we prove the Grim Reaper Theorem.

\begin{corollary} \label{uniform}
Let $\epsilon>0$ be given and
let $J \subset (0,\pi)$ be any closed interval. We have
$$\sup_{\theta \in J} \bigg| \frac{F_{\theta}(\theta,t)}{F(\theta,t)}-
\frac{\cos(\theta)}{\sin(\theta)}\bigg|<\epsilon,$$
for $t$ sufficiently close to $T$.
\end{corollary}

\begin{proof}
We can replace $\kappa$ by $F$ because for each time
these functions are constant multiples of each other.
\end{proof}

Consider the new function
\begin{equation}
G(\theta,t)=\frac{F(\theta,t)}{\sin(\theta)}.
\end{equation}
Using Lemma \ref{uniform} we have the following result:
\begin{equation}
|G_{\theta}| = \frac{|F_{\theta}(\theta,t)\sin(\theta)-F(\theta,t)\cos(\theta)|}{\sin^2(\theta)} <
\frac{\epsilon F(\theta,t) \sin(\theta)}{\sin^2(\theta)}=
\epsilon G,
\end{equation}
This holds for all $\theta \in J$ provided 
that we take $t$ sufficiently close to $T$.
The last calculation shows that the
logarithmic derivative $G_{\theta}/G$ is nearly $0$
on $J$.  Hence $G$ is nearly constant on $J$. But $G(\pi/2,t)=1$.
Hence $G$ is nearly $1$ on $J$.  This proves
that $F(\theta,t)$ converges uniformly to
$\sin(\theta)$ for $t \in J$.  But this combines
with Corollary \ref{uniform} to show that
$F_{\theta}(\theta,t)$ converges uniformly to
$\cos(\theta)$ for $t \in J$.  This completes the proof
of the Grim Reaper Theorem.

\newpage

 \section{Asymptotic Formulas}

  \subsection{The Y Bound}
  
  In this section we deduce the middle bound in
  Equation \ref{AS0} from the Grim Reaper Theorem,
  namely
  $$\liminf_{t \to T} Y(t) \kappa(\pi/2,t) \geq \pi/2.$$
  The key is to get a nice integral formula for this expression.
  
\begin{lemma}
  \begin{equation}
  \label{goal1}
  Y(t) \kappa(\pi/2,t) = \int_0^{\pi/2} \frac{\sin(\phi)}{F(\phi,t)}  d\phi.
\end{equation}
\end{lemma}

\begin{proof}
  Let $s_0$ and $s_1$ respectively denote the arc-length
 parameters that correspond to $\theta_0=0$ and $\theta_1=\pi/2$.
On the level of $1$-forms:
$$dy = -ds \sin \theta, \hskip 30 pt \kappa(\theta,t) ds=d\theta.$$
(The minus sign appears because $y$ decreases as $s$ increases.)
\begin{equation}
  \label{new1}
Y(t)=\int_0^{Y(t)} dy =
-\int_{s_1}^{s_0} \sin(\theta) ds =\int_{s_0}^{s_1} \sin(\theta) ds =
\int_0^{\pi/2} \frac{\sin(\theta)}{\kappa(\theta,t)}d\theta.
\end{equation}
Multiplying through by $\kappa(\pi/2,t)$, we get
Equation \ref{goal1}.
\end{proof}

Letting $\delta>0$ be arbitrary, we have
\begin{equation}
  \label{split}
Y(t) \kappa(\pi/2,t)=\int_0^{\delta} \frac{\sin(\phi)}{F(\phi,t)}
d\phi
+\int_{\delta}^{\pi/2} \frac{\sin(\phi)}{F(\phi,t)}  d\phi >\int_{\delta}^{\pi/2} \frac{\sin(\phi)}{F(\phi,t)}
d\phi
\end{equation}
By the Grim Reaper Theorem, the integrand in the last integral tends
to $1$ as $t \to T$.  Hence the
right hand side is at least $\pi/2-2\delta$ once $t$ is sufficiently
close to $T$.  This establishes our bound.
\newline
\newline
{\bf Remark:\/}  The Y bound in Equation \ref{AS0} is weaker
than the Y bound in Equation \ref{AS5} and one might wonder
about a direct proof of the stronger result.
It is difficult to conclude directly that
the first integral in Equation \ref{split} converges to $0$ as $\delta
\to 0$ because the integrand
could potentially blow up near $\theta=0$.  The issue is that in the Grim Reaper Theorem
we only get convergence on the open interval
$(0,\pi)$. Our indirect argument for the bound in Equation \ref{AS5},
which uses convexity and all the inequalities in Equation \ref{AS0}
together, avoids this difficulty.

\subsection{The X Bound}

The rest of the chapter is devoted to proving
the third bound in Equation \ref{AS0}. A similar 
asymptotic result is proven in \cite{ANG4},
albeit for everywhere locally convex curves.

Define
\begin{equation}
  \beta(t):=\frac{X(t)}{(T-t)\kappa(\pi/2,t)}>0.
\end{equation}
It suffices to show that $\beta(t)>2$ for $t$ sufficiently close to $T$.

Our argument in this section gives a clear reason why this
should be the case, but there is one technical detail
which takes a rather long time to prove.  Here we give the
main argument.

Define
  \begin{equation}
    \ell(t):=\log(X(t)) - \frac{1}{2} \log(T-t).
  \end{equation}

  \begin{lemma}
    \label{step1}
    $\ell_t(t)>0$ if and only if $\beta(t)>2$.
  \end{lemma}

  \begin{proof}
  This is just a calculation.
  We have $X_t(t)=-\kappa(\pi/2,t)$.  Therefore,
  \begin{equation}
    2\ell_t(t)=\frac{2X_t(t)}{X(t)} + \frac{1}{T-t} =
    -\frac{2\kappa(\pi/2,t)}{X(t)} + \frac{1}{T-t}=
    \frac{1}{T-t} \times \bigg(1-\frac{2}{\beta(t)} \bigg).
    \end{equation}
    Hence $\ell_t(t)>0$ if and only if $\beta(t)>2$.
    \end{proof}

    \begin{lemma}
      \label{grow}
      $\lim_{t \to T} \ell(t)=+\infty$.
    \end{lemma}

    \begin{proof}
      This is equivalent to the statement that
$$
  \lim_{t \to T} \frac{X(t)}{\sqrt{T-t}} \to \infty.
$$
Consider the rescaled curve
$C(t)/\sqrt{T-t}$.  The area of this
curve converges to $2\pi$ as $t \to T$
and the aspect ratio converges to $0$.
Hence the rightmost point, namely
$X(t)/\sqrt{T-t}$, converges to $\infty$.
\end{proof}

Since $\ell(t) \to \infty$ as $t \to T$, there is a
sense in which $\ell_t(t)>0$ much more often than
$\ell_t(t) \leq 0$.  However, we don't know
{\it a priori\/} that the sign does not switch
infinitely often as $t \to T$. This is the technical detail.
The rest of the chapter is devoted to showing
that $\ell_t$ changes sign
at most finitely many times as $t \to T$.
This combines with Lemma \ref{grow} to show
that $\ell_t(t)>0$ once $t$ is sufficiently close to $T$.
Lemma \ref{step1} then tells us that $\beta(t)>2$
for $t$ sufficiently close to $T$.

\subsection{The Support Function}

As a prelude to showing that $\ell_t$ changes sign finitely
many times, we discuss some of the geometry of the
curve $C(t)$.

We introduce the {\it support function\/}
\begin{equation}
  p(\theta,t) = C(\theta,t) \cdot {\bf n\/}(\theta), \hskip 30
  pt
  {\bf n\/}(\theta)=(\sin(\theta),\cos(\theta)).
\end{equation}

The normal vector $\textbf{n}$ is the same one as in
Lemma \ref{nega} above.  Again, this vector is independent of time.

\begin{lemma}
  \begin{equation}
    \label{support}
  C(\theta,t) = p(\theta,t) {\bf n\/}(\theta) +
  p_{\theta}(\theta,t) {\bf n\/}_{\theta}(\theta).
  \end{equation}
  Moreover ${\bf n\/}$ is the outward normal vector field with respect
  to
  $C(t)$.
\end{lemma}

\begin{proof}
  This is a classic result.  Since
  ${\bf n\/}$ and ${\bf n\/}_{\theta}$ form
  an orthonormal basis there are functions
  $p(\theta,t)$ and $q(\theta,t)$ such that
  $$C(\theta,t)= p(\theta,t) {\bf n\/}(\theta) + q(\theta,t) {\bf n\/}_{\theta}(\theta).$$
  Note that ${\bf n\/}_{\theta \theta} = -{\bf n\/}$.   Thus, when we differentiate with
  respect to $\theta$, we get
  $$C_{\theta} = (p_{\theta}-q) {\bf n\/} + (p+q_{\theta}) {\bf n\/}_{\theta}.$$
  Since $C_{\theta} \perp {\bf n\/}$ we have
  $$0=C_{\theta} \cdot {\bf n\/} = p_{\theta}-q.$$
  Hence $q=p_{\theta}.$
\end{proof}

\subsection{The Parabolic Rescaling}

The method here is an adaptation of an idea in
Angenent's paper \cite{ANG3}.   See also the paper
\cite{MA} by H. Matano.
We introduce another new variable $\tau$, which is related to $t$ as follows:
\begin{equation}
  \label{TIME}
  \tau = \log \frac{1}{T-t}, \hskip 30 pt
  t=T-e^{-\tau}.
\end{equation}
Note that $\tau \to +\infty$ corresponds to $t \to T$.

We introduce the parabolic rescaling:
\begin{equation}
  D(\theta,\tau)=e^{\tau/2}C(\theta,T-e^{-\tau}).
\end{equation}
Up to changing the time parametrization, this is the
same curve considered in connection with Lemma \ref{nega}.
Next, we introduce the {\it node function\/}
\begin{equation}
  \nu(\theta,\tau) = D_{\tau}(\theta,\tau) \cdot {\bf n\/}(\theta).
\end{equation}
This quantity measures the component of the velocity
of the curve $\tau \to D(\theta,\tau)$ in the normal direction.
Angenent calls points where $\nu(\theta,\tau)=0$ {\it nodes\/}
and proves results about how the number of such is
monotone non-increasing with time.  We take the same approach.

\begin{lemma}
  \label{node}
  For corresponding times $t,\tau$, we have
  $\ell_t(t)=0$ iff
  $\nu(\pi/2,\tau)=0$.
\end{lemma}

\begin{proof}
  We have already mentioned that
  $D(\tau)$ is the same curve as
$C(t)/\sqrt{T-t}$.  
The support function for $D$, namely
the function $P$ considered in Lemma \ref{nega}, is
\begin{equation}
  P(\theta,\tau)=e^{\tau/2} p(\theta,t), \hskip 20 pt
  t=T-e^{-\tau}.
\end{equation}
  The time derivative $P_{\tau}(\pi/2,\tau)$ describes the
  velocity of the point $D(\pi/2,\tau)$. This is zero
  if and only if the velocity of the point
  $$\frac{1}{\sqrt{T-t}} C(\pi/2,t)=\frac{X(t)}{\sqrt{T-t}}$$
  is zero, because $t$ and $\tau$ are related by a diffeomorphism.

  In short, $\ell_t(t)=0$ if and only of
  $P_{\tau}(\pi/2,\tau)=0$.
  Finally, we observe that
  $\nu(\pi/2,\tau)=P_\tau(\pi/2,\tau)$.
\end{proof}

\subsection{Finitely Many Sign Changes}

There are two things that we need to know about
the node function
$\nu$.   We let $K=K(\theta,\tau)$ be the curvature of
$D(\theta,\tau)$. Then:
\begin{equation}
  \label{NU}
 \nu= \frac{P}{2}-K.
\end{equation}

\begin{equation}
  \label{NODAL}
  \nu_\tau=K^2 \nu_{\theta\theta}+(K^2+1/2) \nu.
\end{equation}
We will derive these in the next section.
Equation \ref{NODAL} is a strictly
parabolic equation in the sense of
Equation \ref{GPDE}.

If follows from Lemma \ref{nega} and
Equation \ref{NU}
that $\nu(\theta,\tau)<0$ once $\tau$ is large and
$\theta$ is sufficiently close to $-\alpha(t)$.  Here
$t$ and $\tau$ are corresponding times.
By symmetry, $\nu(\theta,\tau)<0$ when
$\theta$ is sufficiently close to $\pi+\alpha(t)$.
In short, $\nu$ is negative near the boundary of
the spacetime domain on which it is defined.

Since $D$ is analytic, and $\nu$ is negative near the boundary of
the domain, there is a finite number $N(\tau)$ of
points where $\nu(*,\tau)$ vanishes.
By the Sturmian Principle, applied to domains whose
vertical sides are contained entirely in the regions near the boundary
where $\nu<0$, we see that $N(\tau)$ is non-increasing with time and
$N(\tau)$ drops by at least $2$ at any time $\tau$
  where the function $\nu(*,\tau)$ vanishes to at least second order
  at some point.
The function
$\nu(*,\tau)$ is invariant with respect to the
reflection $\theta \to \pi-\theta$.
Hence if this function vanishes at $\pi/2$, it vanishes
to at least second order.  This means that $N(\tau)$ drops
whenever $\nu(\pi/2,\tau)=0$.  Hence this can happen
at most finitely many times.  Lemma \ref{node} now tells us that
$\ell_t(t)$ can vanish at most finitely many times as
$t \to T$.

\subsection{Derivations}

In this section we derive Equations \ref{NU} and \ref{NODAL}.
We need to compute some auxiliary quantities along the way.

\begin{lemma}
  \begin{equation}
    P_{\tau}=\frac{P}{2}-K.
  \end{equation}
\end{lemma}

\begin{proof}
   Using the fact that ${\bf n\/}$ does not depend on time, and is the
   outward normal, we compute
   \begin{equation}
  p_t=\frac{d}{dt} \bigg(C(t) \cdot {\bf n\/}\bigg) = -\kappa {\bf n\/} \cdot
  {\bf n\/}=-\kappa.
  \end{equation}
Now we set $p=p(T-e^{-\tau})$ and use the product and chain rule to compute
$$P_{\tau}=\frac{d}{d\tau}\bigg(e^{\tau/2}
p\bigg)=(1/2) e^{\tau/2}p-e^{-\tau} e^{\tau/2} \kappa=\frac{P}{2}-e^{-\tau/2} \kappa=
\frac{P}{2}-K.$$
This does it.
\end{proof}

\begin{lemma}[Equation \ref{NU}]
$$\nu= \frac{P}{2}-K.$$
\end{lemma}

\begin{proof}
  Suppressing the arguments, we have
  \begin{equation}
    \label{SUPPORT}
    D=P {\bf n\/}+P_{\theta} {\bf n\/}_{\theta}.
  \end{equation}
  Hence
$$\nu=D_{\tau} \cdot {\bf n\/}=(P_{\tau} {\bf n\/} + P_{\theta\tau} {\bf n\/}_{\theta}) \cdot
{\bf n\/} = P_{\tau}=
\frac{P}{2}-K.$$
This does it.
\end{proof}

\begin{lemma}
  \begin{equation}
  \label{SUPP2}
  P+P_{\theta \theta}=\frac{1}{K}.
\end{equation}
\end{lemma}

\begin{proof}
  Up to a different sign convention, this is
  Equation 2.3 in \cite{ANG4}.
  Here is the formula for the signed curvature of a
  parametrized plane curve.
\begin{equation}
  \label{CURV}
K =  \pm \frac{D_{\theta} \times
  D_{\theta\theta}}{\|D_{\theta}\|^3}
\end{equation}
The ambiguity in the sign comes from the fact that
$K$ is always taken to be positive.
Using this equation for the parametrization
given in Equation \ref{SUPPORT},
we get Equation \ref{SUPP2} up to sign.
  To get the sign in Equation \ref{SUPP2} we note that the
sign
is correct in the special case $D$ is the unit circle, parametrized in a clockwise
way.  But then, since we are parametrizing $D$ in a clockwise way,
the sign is correct in the arbitrary case.
\end{proof}

\begin{lemma}
  \begin{equation}
    \label{ANG}
    K_{\tau}=-\frac{K}{2} + K^2K_{\theta \theta} + K^3.
  \end{equation}
\end{lemma}

\begin{proof}
  This is Equation 12 in \cite{ANG}.
   Using the chain rule, and setting $\kappa=\kappa(T-e^{-\tau})$ we
  have
    $$K_{\tau}=\frac{d}{d\tau} e^{-\tau/2} \kappa=$$
 $$ -(1/2)e^{-\tau/2} \kappa - e^{-\tau/2} e^{-\tau} \kappa_t=$$
 $$- (1/2)K + e^{-3\tau/2}(\kappa^2 \kappa_{\theta \theta} +
 \kappa^3)=$$
$$
   -\frac{K}{2} + K^2K_{\theta \theta} + K^3.
   $$
   This does it.
\end{proof}

\begin{lemma}[Equation \ref{NODAL}]
$$  \nu_\tau=K^2 \nu_{\theta\theta}+(K^2+1/2) \nu.
$$
\end{lemma}

\begin{proof}
Differentiating Equation \ref{NU}, we have
$$\nu_\tau=\frac{P_\tau}{2}-K_\tau=$$
$$(P/4-K/2) - (-K/2+K^2K_{\theta \theta} + K^3)=$$
\begin{equation}
-K^2K_{\theta\theta}-K^3+\frac{\nu+K}{2}
\end{equation}

Therefore
$$\nu_{\theta\theta}+\nu=\frac{P_{\theta\theta}+P}{2}-\big(K_{\theta\theta}+K\big)=\frac{1}{2K}-\big(K_{\theta\theta}+K\big)$$
Multiplying through by $K^2$ we get
$$-K^2K_{\theta \theta} - K^3 + \frac{K}{2} = K^2(\nu + \nu_{\theta \theta})$$
Thus:
$$\nu_\tau=-K^2K_{\theta\theta}-K^3+\frac{\nu+K}{2}=K^2(\nu_{\theta\theta}+\nu)+\frac{\nu}{2}=K^2\nu_{\theta\theta}+(K^2+1/2)
\nu.$$
This does it.
\end{proof}

\newpage

\section{The Bowtie Theorem}

The only detail missing in the proof of the
Bowtie Theorem is the Migration Lemma,
which we now prove.

Let $\Gamma(t)=C(t)/X(t)$.
The bounding box for $\Gamma(t)$ is
$$[-1,1] \times [-H(t),H(t)], \hskip 30 pt
H(t)=\frac{Y(t)}{X(t)}.$$

Let $x(P)$ and $y(P)$ respectively denote the $x$ and $y$
coordinates of a point $P$.    Let $\delta_t>0$ denote the value such
that
\begin{equation}
  \label{YY}
  y(\Gamma(\delta_t,t))=H(t)/2.
  \end{equation}

  \begin{lemma}
    \label{rise0}
    There is some
    $\delta>0$ such that $\delta_t>\delta$ once $t$ is
    sufficiently close to $T$.
  \end{lemma}

  \begin{proof}
    The Grim Reaper Curve $G=G(\theta)$ has
 maximum curvature $1$, and it occurs
 at $G(\pi/2)$, a point on the $x$-axis.
 The total height of $G$
is $\pi$.   Hence there is some
value $\delta>0$ such that
$y(G(\delta))= \pi/3.$
Define the rescaling
$$G^*(t)=C(t) \times \frac{\pi/2}{Y(t)}=\Gamma(t) \times \frac{\pi/2}{H(t)}.$$
By the middle formula in Equation \ref{AS5} and the
Grim Reaper Theorem together, $G^*(t)$
converges uniformly to $G$ (modulo horizontal translations) for $\theta \in
[\delta,\pi/2]$.  Hence
$$\lim_{t \to T} y(G^*(\delta,t)) = \pi/3, \hskip 30 pt
\lim_{t \to T} \frac{y(G^*(\delta,t))}{\pi/2} = 2/3.$$
The second equation is just a reformulation of the first.
Recaling, we have
$$\lim_{t \to T} \frac{y(\Gamma(\delta,t))}{H(t)}= 2/3.$$
But then $\delta_t>\delta$ for $t$ sufficiently close to $T$.
\end{proof}

\begin{lemma}
  \label{rise}
  $\lim_{t \to T} x(\Gamma(\delta_t,t))=1$.
\end{lemma}

\begin{proof}
  Let $\delta<\delta_t$ be as in the previous lemma.
By convexity, the arc of $\Gamma(t)$ connecting
  $\Gamma(\delta,t)$ to $\Gamma(\pi/2,t)=(1,0)$
  lies inside the solid right
  triangle bounded by
   the $x$-axis, the
   tangent line to $\Gamma(t)$ at $\Gamma(\delta,t)$,
   and the vertical line through $\Gamma(\delta,t)$.
   But the horizontal side of this triangle has length
   $H(t)/\tan(\delta)$.  Hence
   $$x(\Gamma(\delta_t,t))>x(\Gamma(\delta,t))>1-\frac{H(t)}{\tan(\delta)},$$
   a quantity which tends to $1$ as $t \to T$.
\end{proof}
  
If the Migration Lemma is false, there is some $\eta>0$ and a sequence
of times
$t_n \to T$ such that
\begin{equation}
  \label{MIG0}
  \lim_{n \to \infty} L_{t_n}(C(0,t_n))=(1-\eta,1).
\end{equation}
Let $\delta_n=\delta_{t_n}$ be as in Equation \ref{YY}.
  If we scale the $y$-coordinate by $1/H(t_n)$ (and do
  nothing to the $x$-coordinate) we map
  $\Gamma(t_n)$ to $L_{t_n}(C(t_n))$.
  Therefore, Lemma \ref{rise} and Equation \ref{YY} together give
  \begin{equation}
\label{MIG11}
\lim_{n \to \infty} L_{t_n}(C(\delta_n,t_n))=(1,1/2).
\end{equation}
Combining Equation \ref{MIG0}, Equation \ref{MIG11},
convexity, and symmetry, we see that the right lobe of
$L_{t_n}(C(t_n))$ bounds a convex polygon which converges in
the Hausdorff metric to
the polygon with vertices
$$(0,0), \hskip 10 pt
(1-\eta,1), \hskip 10 pt
(1,1/2), \hskip 10 pt
(1,-1/2), \hskip 10 pt
(1-\eta,-1).$$
But this polygon has area $1+(\eta/2)$.
This contradicts the fact that the area
of the region bounded by the right lobe of
$L_{t_n}(C(t_n))$ converges to $1$.


\newpage

\begin{thebibliography}{15}

\bibitem{AL}
Abresch, U., Langer, J.
\textit{The Normalized Curve Shortening Flow and Homothetic Solutions}.
J. Diff. Geo. 23. (1986) 175-196.

\bibitem{ANG}
Angenent, S., 
\textit{On the Formation of Singularities in the Curve Shortening Flow}.
J. Diff. Geo. 33. (1991) 601-633.

\bibitem{ANG2}
Angenent, S., 
\textit{The Zero Set of a Solution of a Parabolic Equation}.
J. Reine Angew. Math. 390. (1988) 79-96

\bibitem{ANG3}
Angenent, S., 
\textit{Parabolic Equations for Curves on Surfaces (Part II)}.
Annals of Math. Vol. 133. No.1. pp. 171-215. (1991).

\bibitem{ANG4}
Angenent, S. Velazquez, J.J.,
\textit{Asymptotic Shape of Cusp Singularities in Curve Shortening}.
Duke Math. J. Vol. 77. No. 1. (1995).

\bibitem{FM}
Friedman, A., McLeod, B.,
\textit{Blow-up solutions of nonlinear degenerate parabolic equations},
Arch. Rational. Mech. Anal {\bf 96\/} (1986)  55-80

\bibitem{M}
Coiculescu, M.P.,
\textit{Some New Results in Geometric Analysis}.
Mathematics Theses and Dissertations. Brown Digital Repository. Brown University LIbrary. (2021)
https://doi.org/10.26300/6fjz-ax05

\bibitem{DHW}
  Drugan, G., He, W., Warren, M.W.,
  \textit{Legendrian curve shortening in ${\bf R\/}^3$}
  Commun. Anal. Geo. Vol. 26. No. 4. (2018) pp. 759-785.
  
\bibitem{EW}
  Epstein, C. L. and Weinstein, M. I.,
  \textit{A Stable Manifold Theorem for the Curve Shortening Equation},
Communications in Pure and applied Mathematics, {\bf 40\/} (1) (1987) pp. 119-139
  
  \bibitem{EV}
  Evans, L.C.,
  \textit{Partial Differential Equations}.
  2nd Edition. Graduate Studies in Mathematics. 
  Vol. 19. American Mathematical Society. (2010)

 \bibitem{GH}
 Gage, M., Hamilton, R.S.,
 \textit{The Heat Equation Shrinking Convex Plane Curves}.
 J. Diff. Geo. 23. (1986) 69-96.

\bibitem{G1} Gage, M., {\it An Isoperimetric inequality with applications to curve shortening\/}
  Duke Math J. {\bf 50\/} no. 4 (1983) pp 1225 1229
  
\bibitem{G2} Gage, M., {\it Curve Shortening Makes Convex Curves Circular\/},
  Invent Math. {\bf 76\/} (1984) 357-364.

 \bibitem{G}
  Grayson, M.A.,
  \textit{The Heat Equation Shrinks Embedded Curves to Round Points}.
  J. Diff. Geo. Vol. 26, No.2. (1987)
  
  \bibitem{G3}
  Grayson, M.A., 
  \textit{The Shape of a Figure-Eight under the Curve Shortening Flow}.
  Invent. Math. 96, (1989) 177-180.

\bibitem{HH}
  Halldorsson, H.P.,
  \textit{Self-Similar Solutions to the Curve Shortening Flow},
  Transactions of the American Mathematical Society. Vol. 364. No. 10. (2012)
  
\bibitem{MA}
Matano, H.,
\textit{Convergence of Solutions of One Dimensional Semilinear Parabolic Equations},
J. Math. Kyoto Univ. Vol. 18, Article no. 2, (1978) pp. 221-227.
  
  \bibitem{S}
    Sturm, C.,   \textit{M\'{e}moire sur une classe d'\'{e}quations \`{a} diff\'{e}rences partielles}.
  J. Math. Pures Appl. (1836) 373-444.
  
\end{thebibliography}
\end{document}